\documentclass[a4paper]{amsart}
\usepackage{amsmath,amsthm,amssymb,latexsym,epic,bbm,comment,mathrsfs}
\usepackage{graphicx,enumerate,stmaryrd, xcolor, color}
\usepackage[all,2cell]{xy}
\xyoption{2cell}

\usepackage{soul}

\usepackage{calligra}
\DeclareMathAlphabet{\mathcalligra}{T1}{calligra}{m}{n}

\theoremstyle{plain}
\newtheorem{thm}{Theorem}
\newtheorem*{thm*}{Theorem}

\newtheorem{lem}[thm]{Lemma}
\newtheorem{prop}[thm]{Proposition}
\newtheorem{conj}[thm]{Conjecture}
\newtheorem{cor}[thm]{Corollary}
\newtheorem{df-prop}[thm]{Definition-Proposition}

\theoremstyle{definition}

\theoremstyle{remark}
\newtheorem{rem}[thm]{Remark}

\usepackage[all]{xy}
\usepackage[active]{srcltx}
\usepackage[parfill]{parskip}
\usepackage{enumerate}

\newcommand{\Hom}{\operatorname{Hom}}

\usepackage{hyperref}
\newcommand{\mc}{\mathcal}
\newcommand{\mf}{\mathfrak}
\newcommand{\C}{\mathbb C}

\newcommand{\oa}{{\bar 0}}
\newcommand{\ob}{{\bar 1}}
\newcommand{\vare}{\epsilon} 
\def\gl{\mathfrak{gl}}
\newcommand{\g}{\mathfrak{g}}
\def\Ann{{\text{Ann}}}
\def\la{\lambda}
\def\pn{\mf{pe} (n)}

\newcommand{\h}{\mathfrak{h}}

\newcommand{\Z}{{\mathbb Z}}

\def\mod{\operatorname{-mod}\nolimits}

\def\Hom{\operatorname{Hom}\nolimits}

\def\Res{\operatorname{Res}\nolimits}
\def\Ind{\operatorname{Ind}\nolimits}

\def\gl{\mathfrak{gl}}

\def\la{\lambda}

\def\pn{\mf{pe} (n)}

\newcommand{\ad}{\mathrm{ad}}

\newcommand{\fp}{\mathfrak{p}}
\newcommand{\fu}{\mathfrak{u}}
\newcommand{\Real}{\mathrm{Re}}
\newcommand{\Coind}{{\rm Coind}}
\newcommand{\coker}{{\rm coker}}
\newcommand{\op}{{\text{op}}}
\newcommand{\DGD}{{\bf D}{\bf G}^2{\bf D}}

\begin{document}
\title[Serre functors for Lie superalgebras]{Serre functors for Lie superalgebras \\ and 	tensoring with $S^{\mathrm{top}}(\mathfrak{g}_{\overline{1}})$}

\author[Chen]{Chih-Whi Chen} \address{C.-W. C.: Department of Mathematics, National Central University, Chung-Li, Taiwan 32054, and  National Center of Theoretical Sciences, 	Taipei, Taiwan 10617} \email{cwchen@math.ncu.edu.tw}

\author[Mazorchuk]{Volodymyr Mazorchuk}
\address{V. M.: Department of Mathematics, Uppsala University, Box. 480,
	SE-75106, Uppsala \\ \mbox{SWEDEN}} \email{mazor\symbol{64}math.uu.se} 
\date{}

\begin{abstract}	 	We show that the action of the Serre functor on the subcategory of projective-injective modules in a parabolic BGG
		category $\mathcal O$ of a quasi-reductive finite dimensional Lie superalgebra is given by tensoring with the top component of the symmetric power of the odd part of our superalgebra. As an application, we determine, for all strange Lie suepralgebras, when the subcategory of projective injective modules in the parabolic category $\mathcal O$ is symmetric.
 \end{abstract}

\maketitle




\section{Introduction} \label{sect::intro}
\subsection{Motivation}  \label{sect::intro::1}
Let $\mathscr C$ be a $\mathbb{C}$-linear additive category  with finite dimensional morphism spaces. A {\em  Serre functor}  on $\mathscr{C}$, as defined in \cite{BoKa}, is an additive equivalence  $\mathbb{S}$ of $\mathscr{C}$ together with isomorphisms $$\Hom_{\mathscr C}(X, \mathbb{S}Y)\cong \Hom_{\mathscr C}(Y, X)^\ast,$$ for all $X, Y\in \mathscr C$, natural in $X$ and $Y$. If exists, a Serre functor is unique (up to isomorphism) and commutes with all auto-equivalences of $\mathscr{C}$. For example, let $ \mathscr C= \mc D^{b}(A)$ be the bounded derived category of complexes of modules over a finite dimensional associative algebra $A$ of finite global dimension. Then $\mc D^{b}(A)$  admits a Serre functor given by the left derived functor of the {\em Nakayama functor} $A^\ast\otimes{}_-: A\mod \rightarrow A\mod$, where $A^\ast$ is the dual of the regular bimodule, as introduced in \cite{Ha88}. An important infinite-dimensional setup in which Serre functors exist is that of the so-called {\em strongly locally finite categories}, as defined in \cite{MM},
	for which all injective modules have finite projective dimension.   In this case,  the left derived functor  of the  Nakayama functor $\mathbf{N}:= \mc C^\ast\otimes_{\mc C}{}_-:\mc C\mod\rightarrow\mc C\mod$ gives rise to a Serre functor on the full subcategory $\mathscr P(\mc C)$ of  the derived category $\mc D^-(\mc C)$ consisting of all finite complexes of projective objects (the so-called {\em perfect complexes}), see \cite[Subsection 2.3]{MM}.

 This paper is motivated by the ideas and the results of the paper \cite{MM}, which studies the Serre functors for the parabolic category $\mc O^{\mf p}$ associated to semisimple Lie algebras and classical Lie superalgebras $\g=\g_\oa\oplus \g_\ob$ of basic classical or queer type. In \cite{MM}, the authors develop a realization of the Serre functor $\mathbb S$ on $\mathscr{P}(\mc O^{\mf p})$ in terms of Harish-Chandra bimodules, under the assumption that the dimension $\dim \g_\ob$ of the odd part $\g_\ob$ of the Lie superalgebra $\g$ is even.  Under this assumption, it is shown  in \cite{MM} that the
 	endomorphism algebra of a basic additive generator of
 	the subcategory of  projective-injective modules in the
 	parabolic category $\mc O$ is symmetric.   The idea behind the proof is to investigate under what situation the corresponding Nakayama functor is isomorphic to the identity functor when restricted to   projective-injective objects;  see \cite[Proposition 2.3]{MM}.

  	 The present paper is an attempt to understand   the Serre functors for any quasi-reductive Lie superalgebras and the symmetric structure on the full subcategory of $\mc O^{\mf p}$ consisting of all projective-injective objects in full generality.   To be more precise, we give a general realization of the {\em Nakayama functor}  $\mathbf{N}$ on the parabolic category $\mc O^{\mf p}$ in terms of the Harish-Chandra $(\g, \g_\oa)$-bimodules, for any quasi-reductive Lie superalgebras $\g$. In particular, our results apply to all {\em strange} Lie superalgebras, including the {\em periplectic} Lie superalgebras $\pn$ and the {\em queer} \mbox{Lie superalgebras $\mf q(n)$.}

\subsection{Description of main results} \label{sect::intro::2} Let $\g$ be an arbitrary (not necessarily simple) quasi-reductive  finite dimensional Lie superalgebra with a fixed triangular decomposition $\g =\mf n^-\oplus \mf h\oplus \mf n^+$ and the Borel subalgebra $\mf b = \mf h \oplus \mf n^+$.  Associated to this triangular decomposition, we have the corresponding BGG category $\mc O$, see \cite{Ma14, CCC}. By \cite[Corollary 2.11]{CCC}, the top symmetric power $S^{\rm top}(\g_\ob)$ of $\g_\ob$,  which is, by definition, only a $\g_\oa$-module, is, in fact, the restriction of a $\g$-module. Furthermore, the  functor  $$S^{\rm top}(\g_\ob)\otimes({}_-): \mc O\rightarrow \mc O$$  defines an auto-equivalence of $\mc O$ with the quasi-inverse $E_{\g}\otimes({}_-):\mc O\rightarrow \mc O$, where 
$E_\g$ is the one-dimensional $\g$-module determined by the property that  $S^{\rm top}(\g_\ob)\otimes E_\g$ is isomorphic to the  trivial (even) $\g$-module. For example, if $\g$ is either a basic classical or a Q-type Lie superalgebra from Kac's list \cite{Ka1} (see also Subsection \ref{sect::21}), then $E_\g\otimes ({}_-)\cong \Pi^{\dim \g_\ob}({}_-)$, where $\Pi$ denotes the parity change functor.

 For a parabolic subalgebra $\mf p$ of $\g$ containing $\mf b$, we have the corresponding parabolic category $\mc O^{\mf p}$. Let $\mc O^{\mf p}_{\rm int}$ be the  full subcategory of $\mc O^{\mf p}$ consisting of all modules with integral weights.	Set $\mc{PI}^{\mf p}$ and $\mc{PI}^{\mf p}_{\rm int}$ to be the full subcategories of $\mc O^{\mf p}$ and $\mc O^{\mf p}_{\rm int}$ consisting of all projective-injective objects, respectively. The following is our first  main result:
 \begin{thm} \label{main::thm} Let $\g$ be a quasi-reductive Lie superalgebra with a parabolic subalgebra $\mf p$ of $\mf g$.   
Let $\mathbf{N}$ be the  Nakayama functor on $\mc O^{\mf p}$.   Then, we have  $$\mathbf{N}(P)\cong E_\g \otimes P,$$ for any projective-injective module $P\in \mc O^{\mf p}$.    In particular, the following are equivalent:
	\begin{itemize}
		\item[(a)]\label{conda} $\mc{PI}^{\mf p}_{\rm int}$ is symmetric.
		\item[(b)]\label{condb} 	The restriction of  the Serre functor $\mathbb S$ to  $\mathscr P(\mc{PI}_{\rm int}^{\mf p})$ is isomorphic to the identity.
		\item[(c)]\label{condc} 	The module $E_\g$ is isomorphic to the trivial module.
	\end{itemize} 
\end{thm}

Note that (a properly simplified version of) Theorem \ref{main::thm}
	was proved in \cite{MM} under   the additional assumptions 	that $\g$ is either basic classical or of Q-type and the 	dimensional of $\g_\ob$ is even, see \cite[Theorem~5.9, Lemma~5.10]{MM}.  As a special case, it follows that $\mc{PI}^{\g} =\mc{PI}_{\rm int}^{\g}$ is symmetric if and only if $E_\g$ is trivial. For the general linear  Lie superalgebra $\mathfrak{gl}(m|n)$, this was proved in \cite{BS12}.

We propose the following conjecture:
\begin{conj} \label{main::coj} Let $\g$ be an arbitrary quasi-reductive Lie superalgebra
		and $\mf p$ a parabolic subalgebra of $\g$. 
 Then $\mc{PI}^{\mf p}$ is symmetric if and only if $S^{\rm top}(\g_\ob)$ is isomorphic to the trivial $\g$-module. 
\end{conj}

This conjecture is motivated by
	the combination of Theorem \ref{main::thm}  and~\cite[Theorem 4.4]{CCC}, 	where it is shown that, for  any
	simple module $L$ in $\mc O^{\mf p}$, if the indecomposable projective 	cover $P_L$ of $L$ in $\mc O^{\mf p}$ happens to be injective, then the socle of 	$P_L$ is isomorphic to $S^{\rm top}(\g_\ob)\otimes L$. Therefore, $\mc{PI}^{\mf p}$ is not symmetric unless $S^{\rm top}(\g_\ob)$ is the trivial $\g$-module.

 The following theorem, which is our second main result, gives an affirmative answer to \cite[Conjecture 5.14]{MM}:
\begin{thm} \label{thmq::2} Suppose that $\g$ is the queer Lie superalgebra  $\mf q(n)$ (See Subsection~\ref{subsect::qn} for the definition). Let $\mf p$ be a parabolic subalgebra of $\g$. Then, the functor $\Pi^n$ is a Serre functor on $\mathscr{P}(\mc{PI}^{\mf p})$.
\end{thm}
 Our third main result is the following:
\begin{thm} \label{thmp::3} Suppose that $\g  = [\pn,\pn]$ is the simple  Lie superalgebra of type P (See Subsection~\ref{subsect::pen} for the definition). Let $\mf p$ be a parabolic subalgebra of $\g$. Then we have 	\begin{itemize} 	\item[(i)] The Serre functor on $\mathscr{P}(\mc{PI}^{\mf p}_{\rm int})$ is isomorphic to  $\Pi^n$. 	\item[(ii)] $\mc{PI}^{\mf p}_{\rm int}$ is symmetric if and only if $n$ is even. 		\end{itemize} \end{thm}

\vskip0.2cm

\subsection{Structure of the paper.}
The paper is organized as follows. In Section \ref{sect::prel}, we provide some background materials on the representation categories of  quasi-reductive Lie superalgebras and collect general technical results that are to be used in the remainder of the paper. In Section \ref{sect::3}, we provide three realizations of the Nakayama functors for any quasi-reductive Lie superalgebras. The proof of Theorem \ref{main::thm} is given in \mbox{Subsection~\ref{sect::33}}. Finally, we focus on the P-type and Q-type Lie superalgebras in Section \ref{sect::pnqn}. The proofs of Theorem \ref{thmq::2} and \ref{thmp::3} are given in Subsections \ref{subsect::pen} and \ref{subsect::qn}, respectively.

\vskip 0.5cm
 {\bf Acknowledgment.} The first author is partially supported by National Science  and Technology Council grants of the R.O.C., and further acknowledge support from the National Center for Theoretical Sciences. The second author is partially supported by the Swedish Research Council.
\vskip 0.5cm
\section{Preliminaries} \label{sect::prel}
 Throughout the paper, we fix the field of complex numbers $\C$ as the ground field. We will always work with superalgebras, superspaces and supermodules.  Morphisms in the category of superspaces are assumed to preserve the $\Z_2$-grading, and the same thus holds for morphisms of superalgebras or modules over superalgebras. Unless mentioned otherwise, we consider left (super)modules.   
 For any Lie superalgebra $\g =\g_\oa\oplus \g_\ob$, we denote its universal enveloping algebra by $U = U(\g)$.  Similarly, we denote the universal enveloping algebra of $\g_\oa$ by $U_0= U(\g_\oa)$, which a superalgebra concentrated in degree zero. 
 In the following we denote by $U$-Mod and $U_0$-Mod the category of all $\g$- and $\g_\oa$-(super)modules.   We denote by $\Pi({}_-): U\text{-Mod}\rightarrow U\text{-Mod}$   the parity change functor, that is, for any $\g$-module $V =V_\oa\oplus V_\ob$ we have $(\Pi V)_\oa =V_\ob,~ (\Pi V)_\ob =V_\oa.$  

 \subsection{Induction, coinduction and restriction functors} \label{sect::21}
 In the paper, we assume that $\mf g$ is a quasi-reductive Lie superalgebra, that is, $\g_\oa$ is a reductive Lie algebra and $\g_\ob$ is semisimple  under the adjoint action of $\g_\oa$. We have the usual restriction functor $\Res({}_-)$   from  $U$-Mod to $U_0$-Mod. Furthermore, $\Res({}_-)$ has both left and right adjoints, namely, 
 $$\Ind({}_-):=U\otimes_{U_0}{}_-\qquad\mbox{and}\qquad \Coind({}_-):=\Hom_{U_0}(U,{}_-),$$ respectively.
 By  \cite[Theorem~2.2]{BF}, we have 
 \begin{equation}\label{eqBF}\Ind({}_-)\;\cong\; \Coind(S^{\rm top}(\g_\ob)\otimes {}_-).\end{equation}  
  As mentioned in Subsection \ref{sect::intro::1}, the top symmetric power $S^{\rm top}(\g_\ob)$ of $\g_\ob$ is the restriction of a $\g$-module, which we will denote by the same expression. It follows by \cite[Proposition~6.5]{Ka88} that 
 	\begin{equation}\label{eq::SInd}
 		(S^{\rm top}(\g_\ob)\otimes{}_-)\circ\,\Ind \cong\Ind\circ\, (S^{\rm top}(\g_\ob)\otimes{}_-).
 	\end{equation}
 
 Recall Kac's list \cite{Ka1} of the {\em basic classical, $Q$-type and $P$-type} Lie superalgebras:  \vskip0.2cm
 	\leftline{$\text{(Basic classical)}~\mathfrak{gl}(m|n),\,\,\mathfrak{sl}(m|n),\,\,\mathfrak{psl}(n|n),\,\,\mathfrak{osp}(m|2n),\,\, D(2,1;\alpha),\,\, G(3),\,\, F(4),$}
 	\vskip0.2cm
 	\leftline{$\text{(Q-type)}\hskip0.4cm\mathfrak{q}(n),\,\,\mathfrak{sq}(n),\,\, \mathfrak{pq}(n)\mbox{ and }\mathfrak{psq}(n),$} 
 	\vskip0.2cm
 	\leftline{$\text{(P-type)}\hskip0.4cm\pn,\,\,[\pn,\pn],$}  see also \cite{ChWa12, Mu12}. The Q-type and P-type superalgebras in this list are usually referred to as {\em strange} superalgebras. When $\g$ is basic classical or Q-type, we have $\Coind({}_-)\cong \Pi^{\dim \g_\ob}\circ \Ind({}_-)$, see \cite{BF, Go00, Ge98}.

\subsection{Parabolic category $\mc O^{\mf p}$}  \label{sect::22}
We follow the notion of parabolic decompositions of superalgebras from \cite{DMP,Ma14}.  First, we  fix a Cartan subalgebra $\h_{\oa}$ of $\mf g_{\oa}$. Denote by  $\Phi\subset\mf h_{\oa}^\ast$ the set of roots of $\g$. For a root $\alpha \in \Phi$, we denote by  $\g^\alpha$   the   root space associated with $\alpha$, i.e.,  
 	$\mf g^\alpha=\{X\in\mf g\,|\, [h,X]=\alpha(h)X,\,\mbox{ for all $h\in\mf h_{\oa}$}\}.$ For a given vector $H\in\h_{\oa}$, we can define the following subalgebras of $\mf g$:
 	\begin{equation*}\label{deflu}\mf l:=\bigoplus_{\Real \alpha(H)=0} \mf g^\alpha,\quad \mf u^+:=\bigoplus_{\Real \alpha(H)>0} \mf g^\alpha, \quad \mf u^-:=\bigoplus_{\Real \alpha(H)<0} \mf g^\alpha,\end{equation*} 
 	where $\Real(z)$ denotes the real part of $z\in \mathbb C$. We can arrange these subalgebras into a {\em parabolic  decomposition} of $\g$ as follows: $$ \g=\mf u^-\oplus \mf l \oplus \mf u^+.$$  The corresponding  parabolic subalgebra of $\mf g$ is defined as the subalgebra $$\mf p:= \mf l\oplus\mf u^+.$$ We write $\mf p(H)$ for $\mf p$ when it is necessary to keep track of $H\in \h_\oa$. In the case when $\mf l = \g^0$, we write  $\mf n^{\pm}=\fu^{\pm}$ and $\mf h =\mf l$ and call such a decomposition $\mf g = \mf n^-\oplus\mf h\oplus\mf n^+$   a triangular decomposition of $\mf g$. In this way, Borel subalgebras are just parabolic
 		subalgebras associated to triangular decompositions.  We define the Weyl group $W$ of $\mf g$ as the Weyl group of the underlying Lie algebra  $\g_\oa$.  Then it acts naturally on $\h^\ast_\oa$, by definition. Furthermore, the usual dot-action of $W$ on $\h^\ast_\oa$ is defined as $w\cdot \la:= w(\la+\rho) -\rho$, for $w\in W$ and $\la \in \h^\ast_\oa$,  where $\rho$ denotes the half of the sum of all roots in $\mf n_\oa^+$.
  
 Fix a triangular decomposition $\mf g = \mf n^-\oplus\mf h\oplus\mf n^+$ of $\g$. We let $\mc O$ denote the category of finitely generated $\g$-modules  $M$ such that  $M$ is semisimple over $\h_\oa$ and locally finite over $U(\mf n^+)$. This category $\mc O$ can be alternatively defined as  	the category of all $\g$-modules that restrict to the classical category $\mc O^\oa$ for the triangular decomposition $\g_\oa =\mf n^-_\oa \oplus \mf h_\oa \oplus \mf n^+_\oa$ of $\g_\oa$, see \cite{BGG76} and \cite{Hu08}. 
 
Let $\mc F^\oa$ be the full subcategory of $\mc O^\oa$ consisting of all finite dimensional semisimple $\mf g_{\oa}$-modules. We denote by~$\mc F$ the full subcategory of $\mc O$ consisting of all  finite dimensional $\mf g$-modules which restrict to objects in~$\mc F^\oa$. We let  $\Upsilon\subset\mf h^\ast_\oa$ be  the set of all integral weights,  that is, weights appearing in modules in~$\mc F$.

    For $\la\in \h^\ast_\oa$,  we consider the Verma module $\Delta_\la = U_0\otimes_{U(\mf b_\oa)}\C_\la$ without specifying in which parity it is assumed to be. For each coset $\Lambda\in \h^\ast_\oa/\Upsilon$, we denote by $\mc O_\Lambda$ the full subcategory of $\mc O$ consisting of all modules whose weights belong to $\Lambda$. If $\la \in \Lambda$ is a dominant and regular weight, then it follows by \cite[Lemma 2.3]{CC} that the  category of projective objects in $\mc O_\Lambda$ consists of all direct summand of $\g$-modules of the form $V\otimes \Ind\Delta_\la$ with $V\in \mc F$. 
 
  We fix a parabolic subalgebra $\mf p$ of $\mf g$ containing the Borel subalgebra $\mf b$. Then the corresponding parabolic category $\mc O^{\mf p}$ is defined to be the full subcategory of $\mc O$ consisting of all objects   on which the action of $U(\mf p)$ is locally finite. For each $\Lambda\in \h^\ast_\oa/\Upsilon$, we define $\mc O^{\mf p}_\Lambda$ as the intersection of $\mc O_\Lambda$ with $\mc O^{\mf p}$.   
Similarly, the parabolic category $\mc O^{\mf p_\oa}$ of $\g_\oa$-modules and the full subcategory $\mc O^{\mf p_\oa}_\Lambda$ are defined in the same way.  Also, we define the  subcategories $\mc O^{\mf p}_{\rm int} = \mc O^{\mf p}_\Upsilon$, $\mc O_{\rm int} = \mc O_{\Upsilon}$ and $\mc O^{\mf p_\oa}_{\rm int} = \mc O^{\mf p}_\Upsilon$, respectively. Finally,	we let $X_{\mf p_\oa}$ denote the set of all  $\mf p_{\oa}$-dominant weights $\lambda\in\mf h^\ast_\oa$.  	For    $\la \in  X_{\fp_{\oa}}$, we denote by  $\Delta^{\mf p_\oa}_\la\in \mc O^\oa$  the parabolic Verma module over $\g_\oa$  of the highest weight $\la$.

\subsection{Hairsh-Chandra bimodules}  
Let $R,S$ be two superalgebras in the set $\{U,U_0\}$. For an $R$-$S$-bimodule (which is the same thing as a module over $R\otimes_\mathbb{C} S^{\rm op}$) $N$,  let $N^{\ad}$ denote the associated adjoint $U_0$-module. We set  $N^{\g_\oa}\subseteq N^\ad$ to be the submodule of $\g_\oa$-invariants in $N$. We have the left exact functor 
$$\mc L({{}_-},{}_-): (S\mbox{-Mod})^{\op}\times R\mbox{-Mod}\;\to\;R\otimes_{\C}S^{\op}\mbox{-Mod},$$ which takes
the maximal $R$-$S$-submodule of~$\Hom_{\C}(M,N)$ such that~$\Hom_{\C}(M,N)^{\ad}$ is a (possibly infinite) direct sum of modules  in~$\mc F^\oa$. We will use the same notation $\mc L$ for the functor corresponding
	to all possible choices of~$R$ and~$S$.

We set  $\mc B$ to be the category of finitely generated $U$-$U_0$-bimodules $N$ for which $N^{\ad}$ is a Harish-Chandra $U_0$-$U_0$-bimodule, that is, it is a (possibly infinite) direct sum of modules in~$\mc F^\oa$ and each simple  module appears in $N^\ad$ with finite
multiplicity.  For a two-sided ideal $I\subset U_0$,  let $\mc B(I)$ be the full subcategory of~$\mc B$ consisting of bimodules $X$ such that~$XI=0$.

  For any $M\in \g$-Mod, we denote by 
 \begin{enumerate}
 	\item[$\bullet$] $\mc F\otimes M$ the full subcategory of $U$-Mod  consisting  of  all $\mf g$-modules of the form $V\otimes M$ with~$V\in \mc F$, 
 	\item[$\bullet$]  $\langle \mc F\otimes M\rangle$   the abelian category of subquotients of modules in~$\mc F\otimes M$,
	\item[$\bullet$]    $\coker(\mc F\otimes M)$ the category of modules which are presented by modules in $\mc F\otimes M$, that is, it is the category of all $\g$-modules $X$ that have a presentation of the form
	$$V_1\otimes  M\rightarrow V_2\otimes  M\rightarrow X\rightarrow 0,~\text{where $V_1, V_2\in \mc F$}.$$  
 \end{enumerate}

The idea of the following lemma is taken from
\cite[Theorem 3.1]{CC}. It originates from \cite{MiS97}.
\begin{lem} \label{lem::5CC}
	 Let $N$ be a $\g_\oa$-module and set $I:=\Ann_{U_0}(N)$ to be the    annihilator ideal of $N$. If
		\begin{enumerate}[$\bullet$]
			\item  the natural representation map  $ U_0 \rightarrow  \mc L(N,N)$ is surjective, and \label{ThmMS1a}
			\item the module $N$ is projective in~$\langle \mc F^\oa\otimes N\rangle$, \label{ThmMS1b}
		\end{enumerate}
		then we have an equivalence of categories
		$$({}_-)\otimes_{U_0}N\,:\;\;\; \mc B(I)\,\to\, \coker(\mc F\otimes \Ind N),$$
		with inverse $\mc L(N,{}_-)$.
\end{lem}

 Recall that we fix a parabolic subalgebra $\mf p$ of $\g.$
  Let $\la \in X_{\mf p_\oa}$ be  a regular and dominant weight  and set $\Lambda=\lambda+\Upsilon$.  Denote by~$I_\lambda^{\mf p_\oa}:= \Ann_{U_0}\Delta^{\mf  p_\oa}_\la \subset U_0$ the  annihilator ideal  of $\Delta_\la^{\mf p_\oa}$.     The following lemma extends \cite[Corollary 3.2]{CC} to the parabolic category $\mc O^{\mf p}$: 
\begin{lem} \label{lem::CorEqiv2} 
With notation as above,  we  have mutually inverse equivalences ${}_-\otimes_{U_0}\Delta_\la^{\mf p_\oa}$ and $\mc L(\Delta^{\mf p_\oa}_\la,{}_-)$ between 
	$\mc B(I_\la^{\mf p_\oa})$ and $\mc O^{\mf p}_{\Lambda}.$
\end{lem} 
\begin{proof} Let $\Delta:= \Delta_\la^{\mf p_\oa}$. First, we note that the natural representation map $$U(\g) \rightarrow ~\mc L(\Delta,\Delta)$$ is surjective, since $\Delta$ is a quotient of the projective Verma module $\Delta_\la$, see e.g., \mbox{\cite[6.9~(10)]{Ja83}}. By Lemma \ref{lem::5CC}, it follows that the functors  ${}_-\otimes_{U_0}\Delta$ and $\mc L(\Delta,{}_-)$ are mutually inverse equivalences between $\mc B(I_\la^{\mf p_\oa})$ and   $\coker(\mc F\otimes \Ind \Delta)$. By definition, the latter is the full subcategory of $\mc O^{\mf p}$ consisting of $\g$-modules $M$ that have a presentation of the form
	$$V_1\otimes \Ind \Delta\rightarrow V_2\otimes \Ind \Delta\rightarrow M\rightarrow 0,$$ where $V_1, V_2\in \mc F$. Since projective modules in $\mc O_\Lambda^{\mf p}$ are exhausted by direct summands of modules of the form $\Ind (V\otimes \Delta)\subset \Ind V\otimes \Ind \Delta$, for  $V\in \mc F^\oa$, we may conclude that $\coker(\mc F\otimes \Ind \Delta)=\mc O_\Lambda^{\mf p}$. This completes the proof. 
\end{proof}

\subsection{Symmetric algebras and Serre functors} 
Let $\mc C$ be a $\C$-linear additive category with finite-dimensional morphism spaces.    We follow the infinite-dimensional setup of symmetric algebras introduced in \cite[Subsections 2.3, 2.4]{MM}.  Let   $\mc C\text{-mod}$ and $\text{mod-}\mc C$ denote the corresponding categories of left and right $\mc C$-modules.  Assume that $\mc C$ is {\em strongly locally finite} in the sense \cite[Subsection 2.3]{MM}. This means that    
\begin{enumerate}[$($i$)$]
	\item\label{cond0} $\mathcal{C}$ is {\em basic} in the sense that 
	different objects from $\mathcal{C}$ are not isomorphic;
	\item\label{cond1} for any $X,Y\in \mathcal{C}$, the $\mathbb C$-vector space 
	$\Hom_{\mathcal{C}}(X,Y)$ is finite dimensional;
	\item\label{cond2} for any $X\in \mathcal{C}$, there exist only finitely many
	$Y\in \mathcal{C}$ such that \begin{align*}
		&\text{either }\Hom_{\mathcal{C}}(X,Y)\neq 0~~~\text{or~~}\Hom_{\mathcal{C}}(Y,X)\neq 0;
	\end{align*}
	\item\label{cond4} for any $X\in \mathcal{C}$, the endomorphism algebra
	$\Hom_{\mathcal{C}}(X,X)$ is local and basic.
\end{enumerate} 
  We denote by $({}_-)^\ast$ the natural duality functor between  $\mc C\text{-mod}$ and $\text{mod-}\mc C$. Then $\mc C$ is called {\em symmetric} provided that the $\mc C$-$\mc C$-bimodules $\mc C$ and $\mc C^\ast$ are isomorphic. 

   Recall that $\mathscr{P}(\mc C)$ denotes the full subcategory of $\mc D^-(\mc C)$ consisting of all  perfect complexes.  By definition, an additive auto-equivalence $\mathbb{S}: \mathscr{P}(\mc C)\rightarrow  \mathscr{P}(\mc C)$ is a   Serre functor if there is an   isomorphism 
\[  \Hom_{\mathscr{P}(\mc C)}(X, \mathbb{S} Y)\cong \Hom_{\mathscr{P}(\mc C)}(Y, X)^\ast, \] natural in $X$ and $Y$.  Suppose that all injective $\mc C$-modules are of finite projective dimension.   Then, the Serre functor on $\mathscr{P}(\mc C)$ exists and admits a realization via the derived functor $\mc L \mathbf{N}$ of the  Nakayama functor $\mathbf{N}:= \mc C^\ast\otimes_{\mc C}{{}_-}:\mc C\mod \rightarrow\mc C\mod$ by  \cite[Proposisiton~2.2]{MM}.  The following lemma is taken from \cite[Proposition~2.3]{MM}:
\begin{lem} \label{lem::4MM}
	 Assume that all injective $\mc C$-modules are of finite projective dimension. Then $\mc C$ is symmetric if and only if the Serre functor  on $\mathscr{P}(\mc C)$ is isomorphic to the identity. 
\end{lem}

\section{Realizations of the Nakayama functor $\mathbf{N}$} \label{sect::3}
In this section, we let $\g$ be an arbitrary quasi-reductive Lie superalgebra with a fixed triangular decomposition $\g=\mf n^-\oplus \mf h \oplus \mf n^+$ and a parabolic subalgebra $\mf p\subseteq \mf g$ containing the Borel subalgebra $\mf b = \mf h\oplus \mf n^+$.  

For a given Harish-Chandra bimodule $B$,  denote by $B^\ast$ and $B^\oast$ the dual bimodule and the sub-bimodule of $B^\ast$ consisting of all elements on which the adjoint action of $\g_\oa$ is locally finite, respectively.
\subsection{Realization via Harish-Chandra $(\g,\g_\oa)$-bimodules} \label{sect::31}
The goal of this subsection is to provide a general realization of the Nakayama functor $\mathbf{N}$ on $\mc O^{\mf p}_\Lambda$, partially following the strategies in \cite[Sectoins 4, 5]{MM}.   Fix  a regular and dominant weight $\la \in X_{\mf p_\oa}$ and denote $\Lambda =\la+\Upsilon$. Recall that we denote by~$I_\lambda^{\mf p_\oa}:= \Ann_{U_0}\Delta^{\mf p_\oa}_\la \subset U_0$ the  annihilator ideal  of $\Delta_\la^{\mf p_\oa}$. Also, we let $\eta\in \h_\oa^\ast$ and $i\in \Z_2$ be determined by $S^{\rm top}(\g_\ob) =\Pi^i \C_\eta$, and so $E_\g = \Pi^i\C_{-\eta}$. Here, $\C_{\pm \eta}$ denotes the one-dimensional $\g$-module of weight $\pm \eta$, respectively.

\begin{thm} \label{thm::realNak}
	The functor 
	\begin{align*}
		&E_\g\otimes \mc L({}_-,\Delta^{\mf p_\oa}_\la)^{\oast}\otimes_{U_0} \Delta^{\mf p_\oa}_\la 
	\end{align*} is isomorphic to the Nakayama functor $\mathbf{N}$ on $\mc O_\Lambda^{\mf p}$. In particular, its left derived functor gives rise to a Serre functor on $\mathscr{P}(\mc O_\Lambda^{\mf p})$.
\end{thm}

Before giving the proof of Theorem \ref{thm::realNak}, we need the following analogue of \cite[Lemma~2.2]{BG}, which is taken from the proof of \cite[Theorem 3.1]{CC}:  
\begin{lem} \label{lem::5}
For any $V\in \mc F$, we have 
\begin{align} \label{eq::inlem5}
&\Hom_{U\text{-}U_0}(V\otimes (U/UI_\la^{\mf p_\oa}), X)\cong \Hom_{U_0}(\Res V, X^{\ad}), \text{ for any $X\in \mc B(I^{\mf p_\oa}_\la)$.}
\end{align}
\end{lem}

\begin{proof}[Proof of Theorem \ref{thm::realNak}]   Our goal is to establish some key steps, while omitting the parts that are analogous to the case of basic classical and queer Lie
		superalgebras for which  we refer to the proof of \cite[Theorem 4.1, Corollary 4.4]{MM} for details. In the proof we  set $\Delta:=\Delta_\la^{\mf p_\oa}$ and $I:=I_\la^{\mf p_\oa}$. It follows from \cite[6.9 (10)]{Jo82} that  
	\begin{align}
		&\mc L(\Delta,\Delta)\cong U_0/I. \label{eq::LDDUI}
\end{align}First, we shall show that for any $P, N\in \mc O_\Lambda^{\mf p}$, with $P$ projective, there is an isomorphism 
\begin{align}
	&\Hom_{U\text{-}U_0}(\mc L(\Delta, P), \mc L(\Delta, N)) \cong \mc L(P,\Delta) \otimes_{U\text{-}U_0} \mc L(\Delta, N\otimes S^{\rm top}(\g_\ob)), \label{eq::11}
\end{align} natural in both $P$ and $N$.  Here $\otimes_{U\text{-}U_0}$ denotes the tensor product over $U\otimes U_0^{\text{op}}$.  We set $$N_\eta:=  N\otimes S^{\rm top}(\g_\ob),~\text{ and }~N_{-\eta}:=N\otimes E_\g.$$ To establish \eqref{eq::11}, we first consider the case $P=\Ind \Delta$. We introduce the following induction functor $$\Ind^r({}_-):= {}_-\otimes_{U_0}U:~U_0\otimes U_0^{\op}\text{-Mod}\rightarrow  U_0\otimes U^{\op}\text{-Mod}.$$

 We may observe that,  for any $\g_\oa$-modules $X,Y$, we have 
 \begin{align}
 	&\mc L(X, \Ind Y)\cong U\otimes_{U_0}\mc L(X, Y)\text{ ~and~ }\mc L(\Ind X,  Y)\cong \Ind^r\mc L(X, Y)\otimes E_\g, \label{eq::Colem3}
 \end{align} as $U$-$U_0$- and $U_0$-$U$-bimodules, respectively; see, e.g., \cite[Lemma 3.7]{Co}.   It follows that 
 	$$\begin{array}{rcl}
 		\Hom_{U\text{-}U_0}(\mc L(\Delta, \Ind \Delta), \mc L(\Delta, N))& \overset{\text{by \eqref{eq::Colem3}}}{\cong} &  \Hom_{U\text{-}U_0}(U\otimes_{U_0}\mc L(\Delta, \Delta), \mc L(\Delta, N)) \\
 		&\overset{\text{by \eqref{eq::LDDUI}}}{\cong}  & \Hom_{U\text{-}U_0}(U/UI, \mc L(\Delta, N)) \\
 		&\overset{\text{by \eqref{eq::inlem5}}}{\cong}& \Hom_{U_0}(\C, \mc L(\Delta, N)^\ad)\\
 		&\overset{}{\cong}& \mc L(\Delta, N)^{\g_\oa}.
 \end{array}$$
On the other hand, we calculate 
	$$\begin{array}{rcl}
	\mc L(\Ind \Delta, \Delta)\otimes_{U\text{-}U_0}\mc L(\Delta,N)&\overset{\text{by \eqref{eq::Colem3}}}{\cong} & \Ind^r\mc L(\Delta,\Delta)\otimes E_\g\otimes_{U\text{-}U_0} \mc L(\Delta, N) \\
	&\overset{\text{by \eqref{eq::LDDUI}}}{\cong}&	 U/IU \otimes_{U\text{-}U_0} \mc L(\Delta, N_{-\eta})\\
	&\overset{\text{taking the double dual}}{\hookrightarrow}& \Hom_{\C}(U/IU \otimes_{U\text{-}U_0} \mc L(\Delta, N_{-\eta}), \C)^\ast\\
	&\overset{\text{ by adjunction}}{\cong}& \Hom_{U\text{-}U_0}(U/UI, \mc L(\Delta, N_{-\eta})^\oast)^\ast \\
	&\overset{\text{ by \eqref{eq::inlem5}}}{\cong}& \Hom_{U_0}(\C, (\mc L(\Delta, N_{-\eta})^\oast)^\ad)^\ast \\	 
	&\overset{\text{by \cite[Lemma 4.3]{MM}}}{\cong}& \mc L(\Delta, N_{-\eta})^{\g_\oa}.
	 \end{array}$$Since $\mc L(\Delta, N_{-\eta})^{\g_\oa}$ is finite-dimensional, the inclusion $$U/IU \otimes_{U\text{-}U_0} \mc L(\Delta, N_{-\eta}) \subseteq \Hom_{\C}(U/IU \otimes_{U\text{-}U_0} \mc L(\Delta, N_{-\eta}), \C)^\ast$$ is indeed an isomorphism.   By an argument similar to that used in the proof of \cite[Proposition 4.2]{MM}, we obtain natural isomorphisms
\begin{align}
&\mc L(\Ind \Delta, \Delta)\otimes_{U\text{-}U_0}\mc L(\Delta,N)\cong \mc L(\Delta, N_\eta)^{\g_\oa}.
\end{align} Together with the fact that all projective modules in $\mc O^{\mf p}_\Lambda$ consists of direct summands of $\g$-modules of the form $V\otimes \Ind \Delta$, with $V\in \mc F$, we may conclude that the dimensions of the vector spaces on both the left hand and right hand sides of \eqref{eq::11} are the same, for any $P, N\in \mc O^{\mf p}_\Lambda$ with $P$ projective.  The rest of the proof of \eqref{eq::11}  is analogous to that of \cite[Proposition 4.2]{MM}.  Consequently, we have the following calculation similar to that used in the proof of \cite[Theorem 4.1]{MM}:
	$$\begin{array}{rcl}
	\Hom_{U}(P,N)	&\overset{\text{ by Lemma \ref{lem::CorEqiv2}}}{\cong}&\Hom_{U\text{-}U_0}(\mc L(\Delta, P), \mc L(\Delta, N))\\
	&\overset{\text{ by \eqref{eq::11}}}{\cong}& \mc L(P, N)\otimes_{U\text{-}U_0} \mc L(\Delta, N_{\eta})\\
		&\overset{\text{taking the double dual}}{\cong}& \Hom_{\C}(\mc L(P,\Delta)\otimes_{U\text{-}U_0} \mc L(\Delta, N_{\eta}),\C)^\ast \\
			&\overset{\text{ by adjunction}}{\cong}& \Hom_{{U\text{-}U_0}}( \mc L(\Delta, N_{\eta}), \mc L(P,\Delta)^\oast)^\ast  \\
				&\overset{\text{ by Lemma \ref{lem::CorEqiv2} again}}{\cong}& \Hom_{U_0}(N_{\eta}, \mc L(P,\Delta)^\oast\otimes_{U_0} \Delta)^\ast  \\
					&\overset{}{\cong}& \Hom_{U_0}(N, E_\g\otimes \mc L(P,\Delta)^\oast\otimes_{U_0} \Delta)^\ast. 
	\end{array}$$This proves the first claim of Theorem \ref{thm::realNak}.  Since every injective module in  $\mc O^{\mf p}_\Lambda$ has finite projective dimension, it follows by  \cite[Proposition 2.2]{MM} that the left derived functor of  $\mathbf{N}({}_-)\cong E_\g \otimes \mc L({}_-,\Delta)^\oast\otimes_{U_0} \Delta: \mc O^{\mf p}_\Lambda\rightarrow \mc O^{\mf p}_\Lambda$ is a Serre functor on $\mathscr{P}(\mc O^{\mf p}_\Lambda)$. This completes the proof. 
\end{proof}

 The following proposition is a restatement of the first claim in Theorem \ref{main::thm}. 
\begin{prop} \label{prop::10} Let $P\in \mc O^{\mf p}_\Lambda$ be a projective-injective module. Then,  we have  $$\mathbf{N}(P)\cong E_\g \otimes P.$$\end{prop}\begin{proof}Consider the endo-functor $C: \mc O^{\mf p}_\Lambda\rightarrow \mc O_\Lambda^{\mf p}$ of partial coapproximation with respect to projective-injective modules in $\mc O^{\mf p}_{\Lambda}$. This functor $C$ can be defined as the unique (up to isomorphism) right exact functor that sends a projective module    $P\in \mc O^{\mf p}_\Lambda$ to its submodule generated
		by all possible images in $P$ of all
		projective-injective modules. The functor $C$
		acts on homomorphisms via restriction, see e.g., \cite{KM, MM}. Denote by $\mc O_\Lambda^{\mf p_\oa}$ the full subcategory of $\mc O^\oa$ consisting of all modules $M$ whose support belongs to $\Lambda$
		and the   action of $U(\mf p_\oa)$  on which is locally finite.  If we let $C_0: \mc O^{\mf p_\oa}_\Lambda\rightarrow \mc O^{\mf p_\oa}_\Lambda$ denote the partial co-approximation with respect to projective-injective modules in $\mc O^{\mf p_\oa}_\Lambda$, then there are natural isomorphisms $\Res \circ\, C \cong C_0 \circ\, \Res$ and $\Ind \circ\, C_0 \cong C \circ\, \Ind$ and $C^2$ sends projective modules to injective modules, see e.g., \cite[Subsection 3.3]{Ch}.  If $P\in \mc O_\Lambda^{\mf p}$ is projective, then it follows by \cite[Corollary~4.6]{MM} that $\Res \mathbf{N}(P) \cong E_\g \otimes  C_0^2 (\Res P)\cong E_\g\otimes \Res C^2(P)$.  We may conclude that $\mathbf{N}(P)$ and $E_\g\otimes C^2(P)$ are isomorphic since they are injective modules (see, e.g., \cite[Lemma~3.3(ii), Theorem~4.4]{CCC}). In particular, if $P$ is injective, then we have $\mathbf{N}(P)\cong E_\g\otimes P$. This completes the proof.   
\end{proof}
The following is a direct consequence of  Lemma \ref{lem::4MM} and Proposition \ref{prop::10}. 
\begin{cor} \label{coro::111}
	Suppose that $\g$ is a   quasi-reductive Lie superalgebra with a parabolic subalgebra $\mf p$.  Let  $\Lambda \in \h_\oa^\ast/\Upsilon$.  Then the category $\mc{PI}_\Lambda^{\mf p}$ is not symmetric unless the $\g$-module $S^{\rm top}(\g_\ob)$ is trivial.
\end{cor}


\subsection{Realization via twisting functors}  \label{sect::32}
 We continue to assume that $\g$ is an arbitrary  quasi-reductive  Lie superalgebra  with a fixed triangular decomposition $\g =\mf n^-\oplus \mf h\oplus \mf n^+$  and a parabolic subalgebra $\mf p\subseteq \mf g$ containing $\mf b = \mf h\oplus \mf n^+$. In this subsection, we focus on the integral block $\mc O^{\mf p}_{\rm int} = \mc O^{\mf p}_\Upsilon$ of the parabolic category $\mc O^{\mf p}$.

 We recall the construction of twisting functors from~\cite{Ar97},   see also \cite{Mat00,AS, KM, CM, CC}.   Fix a non-zero root vector $x \in \mathfrak{g}_{\bar{0}}^{-\alpha}$ associated with a simple root  $\alpha$  of $\mf n_\oa^+$.    Then we have the Ore localisation $U'_{\alpha}$ of~$U$ with respect to the set of  all non-negative integer powers of~$x$ since the adjoint action of~$x$ on $\mathfrak{g}$ is nilpotent.  Consider the quotient $U_\alpha:=U'_{\alpha}/U$ of the  $U$-$U$-bimodule $U'_\alpha$ by the sub-bimodule $U$.  
 	Let $s\in W$ be the simple reflection associated with $\alpha$,  and let $\varphi_{\alpha}$ be an (even) automorphism of~$\mathfrak{g}$ that maps $\mathfrak{g}^{\beta}$ to  $\mathfrak{g}^{s(\beta)}$ for all simple roots $\beta$.  Set   $^{\varphi_{\alpha}}U_{\alpha}$  to be the bimodule obtained from $U_{\alpha}$ by twisting the left action of~$U$ by~$\varphi_\alpha$. Then the twisting functor on $\mc O$ is   defined as
 	$$
 	T_{s}({}_-)= T_{\alpha}({}_-):=  ^{\varphi_{\alpha}}U_{\alpha}\otimes - : \mathcal{O} \rightarrow \mathcal{O}.
 	$$   
   Let $w_0$ be the longest element in $W$. Since the twisting functors $T_s$ ($s\in W$) satisfy the braid relations (see \cite[Theorem 2]{KM}), we  have the twisting functor $T_{w_0}$ defined via composition with respect to an arbitrary reduced expression for $w_0$.  
   
     Now, let $\ell({}_-): W\rightarrow \mathbb N$ be the length function of $W$ and $w_0^{\mf p}$ be the longest element of the Weyl group of the  Levi subalgebra of $\mf p$. We consider the cohomology functor $\mc L_{\ell(w_0^{\fp})}T_{w_0}$. It follows from \cite[Theorem~8.1]{CM17} that its restriction to $\mc O^{\mf p}_{\rm int}$, which we should denote by $\bf T$, is a right exact functor from  $\mc O^{\mf p}_{\rm int}$ to $\mc O^{\hat{\mf p}}_{\rm int}$, for some parabolic subalgebra $\hat{\mf p}$ of $\g$, see also \cite[Subsection 3.4]{CCC}. Here, if $\mf p = \mf p(H)$, for some $H\in \h_\oa$, then $\hat{\mf p}= \mf p(-w_0 H)$. We use the notation ${\bf T}_0({}_-): \mc O^{\mf p_\oa}_{\rm int}\rightarrow \mc O^{\hat{\mf p}_\oa}_{\rm int}$ to denote the same cohomology functor as defined above for $\g_\oa$.
   By \cite[Equation~(5.1)]{CM}, we have
   \begin{equation} \label{eq::TIIT} 
   	{\bf T}\circ\,\Ind\cong\Ind\circ\,{\bf T}_0\qquad\mbox{and}\qquad \Res\circ\,{\bf T}\cong{\bf T}_0\circ\,\Res.
   \end{equation} 
     Recall that  $C_0: \mc O^{\mf p_\oa}_{\rm int}\rightarrow \mc O^{\mf p_\oa}_{\rm int}$ denotes the partial co-approximation with respect to projective-injective modules in $\mc O^{\mf p_\oa}_{\rm int}$. Then both functors ${\bf T}^2_0$ and $C_0^2$ are isomorphic to the Nakayama functor $\mathbf{N}_0$ on $\mc O_{\rm int}^{\mf p_\oa}$; see   \cite[Proposition 4.1]{MS08a} and \cite[Corollary~5.12]{MM}.  The following is the main result in this subsection: 
 \begin{prop} \label{prop::twisting}  The functor $(E_\g\otimes{}_-)\circ {\bf T}^2$ is isomorphic to the Nakayama functor on $\mc O_{\rm int}^{\mf p}$. Furthermore, the functor $$(E_\g \otimes {}_-)\circ (\mc  L{\bf T})^2: \mathscr{P}(\mc O_{\rm int}^{\mf p})\rightarrow \mathscr{P}(\mc O_{\rm int}^{\mf p})$$ is a Serre functor on $\mathscr{P}(\mc O_{\rm int}^{\mf p})$.
 \end{prop}
\begin{proof}  
	For any $N\in \mc O_{\rm int}^{\mf p}$ and $Q\in \mc O^{\mf p_\oa}_{\rm int}$ with $Q$ projective, we have the following natural isomorphisms
	  $$\begin{array}{rcl}
  \Hom_{\mc O^{\mf p}}(N, {\bf T}^2 \Ind Q) &\overset{\text{by }\eqref{eqBF}}{\cong} &\Hom_{\mc O^{\mf p}}(N, {\bf T}^2\Coind(S^{\mathrm top}(\g_\ob)\otimes Q)) \\ 
  & \overset{\text{by} \eqref{eq::TIIT}}{\cong} &\Hom_{\mc O^{\mf p}}(N, \Coind ({\bf T}_0^2(S^{\mathrm top}(\g_\ob)\otimes  Q))\\
    & \overset{\text{by adjunction}}{\cong} & \Hom_{\mc O^{\mf p_\oa}}(\Res N, {\bf T}_0^2(S^{\mathrm top}(\g_\ob)\otimes  Q))\\
      & \overset{\text{by } {\bf T}_0^2\cong \mathbf{N}_0}{\cong} &\Hom_{\mc O^{\mf p_\oa}}(S^{\mathrm top}(\g_\ob)\otimes Q, \Res N)^\ast\\
        & \overset{\text{by adjunction and }\eqref{eq::SInd}}{\cong} &\Hom_{\mc O^{\mf p}}( S^{\mathrm top}(\g_\ob)\otimes \Ind Q, N)^\ast \\
          & \overset{}{\cong} &\Hom_{\mc O^{\mf p}}(\Ind Q, E_\g\otimes N)^\ast.
\end{array}$$

This implies that $(E_\g \otimes {}_-)\circ {\bf T}^2$ is isomorphic to the Nakayama functor $\mathbf{N}$ on $\mc O^{\mf p}_{\rm int}$. The conclusion now follows from \cite[Proposition 2.3]{MM}.
\end{proof}
 
 The next result, which proves \cite[Conjecture 5.14]{MM}, is a  consequence of Theorem~\ref{prop::twisting}, but we provide an alternative proof, which is of interest in its own right. 
\begin{cor} Suppose that $\g$ is a basic classical or a Q-type  Lie superalgebra. Then the functor $$\Pi^{\dim \g_\ob}\circ (\mc  L{\bf T})^2: \mathscr{P}(\mc O_{\rm int}^{\mf p})\rightarrow \mathscr{P}(\mc O_{\rm int}^{\mf p})$$ is a Serre functor on $\mathscr{P}(\mc O_{\rm int}^{\mf p})$. In particular, $\Pi^{\dim \g_\ob}$ is a Serre functor on $\mathscr{P}(\mc{PI}^{\g})$.
\end{cor}
\begin{proof}
Set $F:={\bf T}^2:\mc O^{\mf p}_{\rm int}\rightarrow\mc O^{\mf p}_{\rm int}$. We claim that $F$ has the following properties:
\begin{itemize}
	\item[(i)] $\mc LF: \mathscr P(\mc O^{\mf p}_{\rm int})\rightarrow \mathscr  P(\mc O^{\mf p}_{\rm int})$ is an auto-equivalence.
	\item[(ii)] $F$ maps projective modules   to injective modules. 
	\item[(iii)] The restrictions of $(E_\g\otimes {}_-)\circ\,F$ and $\mathbf N$ to $\mc{PI}_{\rm int}^{\mf p}$ are isomorphic.
\end{itemize}
 Assertion (i) follows by \cite[Proposition 5.11]{CM}.  To prove (ii), let $P\in \mc O^{\mf p}_{\rm int}$ be a projective module. Then $P$ is a direct summand of $\Ind \Res P$.   It follows from \eqref{eq::TIIT} that $${\bf T}\Ind \Res P\cong \Ind {\bf T}_0 \Res P.$$ Therefore, we  conclude that $FP$ is an injective module in $\mc O^{\mf p}_{\rm int}$. Finally, by \cite[Propsotion 5.9]{Co} or an argument analogous to the one in  \cite[Corollaries 4.6, 5.12]{MM}, the functor $F$ is isomorphic to the functor $C$ of   partial coapproximation with respect to projective-injective modules in $\mc O^{\mf p}_{\rm int}$. Therefore, Assertion (iii) follows.  Using (i), (ii) and (iii), the rest of the proof follows the proof of \cite[Theorem 3.4]{MS08a} mutatis mutandis. 
\end{proof}

\subsection{Realization via Joseph's Enright completion functors} \label{sect::33}
	We keep the same notation as in the previous section. 
	The aim of this subsection is to give a realization of the Nakayama functor on $\mc O_{\rm int}^{\mf p}$ in term of Joseph's version of Enright completion functors from~\cite{Jo82}, see also \cite{KM, Co, CC}.  	First, we introduce a duality functor ${\bf D}$ on the category $\mc O$ from \cite[Subsection 1.3]{CCC} as follows.  For any $\g$-module $M$,   let $M^\ast$  denote the canonical dual module $M^\ast = M_\oa^\ast\oplus M_\ob^\ast$ endowed with the action given by $x(f)(n)=-(-1)^{|x||f|}f(xn)$, for homogeneous elements  $f\in M^\ast$,  $x\in\g$ and  $m\in M$. Here $|x|$ and $|f|$ stand for the  parities of   $x$ and $f$, respectively.
  Since each inner  automorphism of $\g_\oa$ extends to an automorphism of $\g$ (see e.g., \cite[Subsection 3.1]{Mu12}), the action of $w_0$ defines an automorphism $\phi$ of $\g$.  We twist the dual module $M^\ast$ by the automorphism $\phi$ and denote the new module by $M^\ast_{\phi}$. We define ${\bf D} M$  to be the    maximal submodule of $ M^\ast_{\phi}$ on which $\h_\oa$ acts semisimply and locally finitely. This defines a duality functor on $\mc O$, which restricts to an exact contravariant involutive equivalence  ${\bf D}: \mc O^{\mf p}\rightarrow \mc O^{\hat{\mf p}}$, see \cite[Proposition~3.4]{CCC}.
  
 Let  ${\bf D}_0: \mc O^{\mf p_\oa}\rightarrow \mc O^{\hat{\mf p}_\oa}$ denote the   functor ${\bf D}$ constructed above for $\g_\oa$-modules. 
It follows from \cite[Proposition 2.11]{Ge98} that 
\begin{align} \label{eq::DICD}
	&{\bf D}\circ\,\Ind \cong \Coind\circ\,{\bf D}_0, 
\end{align} 
		  Fix    a dominant,  regular and integral weight $\lambda \in X_{\mf p_\oa}$ and set $\Lambda=\lambda+\Upsilon$.  For any simple reflection $s\in W$, we recall the following completion functor on $\mc O_\Lambda$ from \cite[Subsection 4.2]{CC}, which   is an analogue of Joseph’s version of Enright completion functor  from \cite[Section 2]{Jo82}:
	$$G_s({}_-) := \mc L(\Delta_{s \cdot \la}, {}_-)\otimes_{U_0} \Delta_\la:\; \mc O_{\rm int} \rightarrow \mc O_{\rm int}.$$

	Let $\text{ID}_{\rm int}$ be the identity functor on $\mc O_{\rm int}$. For any $M\in \mc O_{\rm int}$, the embedding $\Delta_{s\cdot \la}\hookrightarrow \Delta_\la$ gives a natural homomorphism $\mc L(\Delta_{\la},M)\rightarrow \mc L(\Delta_{s\cdot \la},M)$.   
	This homomorphism gives rise to a natural  transformation ${\text{ID}_{\rm int}} \rightarrow G_s$, which is an isomorphism when restricted to $\mc{PI}_{\rm int}$, see   \cite[Subsection 2.4]{Jo82} and  \cite[Subsection 2.3]{KM}.

 Similarly, we have the completion functor $G_{w_0}$ defined via composition with respect to a reduced expression for $w_0$.   
We consider the cohomology functor $\mc R_{\ell(w_0^{\fp})}G_{w_0}$. Then its restriction to $\mc O^{\mf p}_{\rm int}$, which we shall denote by ${\bf G}$, is a left exact functor from  $\mc O^{\mf p}_{\rm int}$ to  $\mc O^{\hat{\mf p}}_{\rm int}$. We shall use the notation ${\bf G}_0({}_-): \mc O^{\mf p_\oa}_{\rm int}\rightarrow \mc O^{\hat{\mf p}_\oa}_{\rm int}$ to denote the same cohomology functor  for $\g_\oa$. It follows from \eqref{eq::Colem3} that
\begin{align}\label{eq::GI}
&{\bf G}\circ\, \Ind \cong \Ind \circ\, {\bf G}_0,~  \Res \circ\,{\bf G} \cong {\bf G}_0 \circ\, \Res
\end{align}
In the case when $\g = \g_\oa$ is  a (reductive) Lie algebra, it  was proved in \cite[Theorem~3]{AS} that $G_s$ is right adjoint to $T_s$. Furthermore, each of them is a conjugation of the other by 	the natural duality on $\mc O_{\rm int}$ by \cite[Theorem 4.1]{KM}. In particular,  ${\bf D}_0 {\bf G}_0^2 {\bf D}_0$  is isomorphic to the Nakayama functor ${\mathbf N}_0$ on $\mc O_{\rm int}^{\mf p_\oa}$.  These remain valid for basic classical and P-type Lie superalgebras, see  \cite[Thereom 5.5]{Co} and \cite[Theorem 4.5]{CC}. The following is a generalization to arbitrary quasi-reductive Lie superalgebras.

  \begin{prop} \label{prop::completion}  
  	We have an isomorphism of endofunctors on $\mc O_{\rm int}^{\mf p}$: $$\DGD\cong {\bf T}^2.$$ 
  	In particular, there is a natural transformation ${\bf T}^2\rightarrow {\emph{ID}_{\rm int}}$, which  is an isomorphism when restricted to $\mc{PI}_{\rm int}^{\mf p}$.  
 \end{prop}
  \begin{proof}   
  	For   $N\in \mc O_{\rm int}^{\mf p}$ and $Q\in \mc O^{\mf p_\oa}_{\rm int}$ with $Q$ projective, we have   natural isomorphisms 
  $$\begin{array}{rcl}
  	\Hom_{\mc O^{\mf p}}(N,  \DGD \Ind Q) & \overset{\text{by }\eqref{eq::DICD},~\eqref{eq::GI}}{\cong} & \Hom_{\mc O^{\mf p}}(N, \Ind({\bf D}_0{\bf G}_0^2{\bf D}_0 Q)) \\
  		&\overset{\text{by }\eqref{eqBF}}{\cong} &  \Hom_{\mc O^{\mf p}}(N, \Coind(S^{\rm top}(\g_\ob)\otimes {\bf D}_0{\bf G}_0^2{\bf D}_0 Q)) \\
  		&\overset{\text{by adjunction}}{\cong}&\Hom_{\mc O^{\mf p_\oa}}(\Res N\otimes E_\g, {\bf D}_0{\bf G}_0^2{\bf D}_0  Q) \\ 
  			& \overset{\text{by ${\bf D}_0 {\bf G}_0^2 {\bf D}_0\cong {\bf N}_0$}}{\cong} & \Hom_{\mc O^{\mf p}}(  Q, \Res  N\otimes E_\g)^\ast\\ 
  			& \overset{\text{by adjucntion}}{\cong} &  \Hom_{\mc O^{\mf p}}(  \Ind Q, E_\g\otimes N)^\ast.  
  	\end{array}$$ 
  	
   This implies that $(E_\g\otimes {}_-)\circ \DGD$ is isomorphic to the Nakayama functor on $\mc O^{\mf p}_{\rm int}$.  The conclusion now follows from Proposition \ref{prop::twisting}.
  \end{proof}

\begin{proof}[Proof of Theorem \ref{main::thm}]
	The first assertion of Theorem \ref{main::thm} is proved in Proposition \ref{prop::10}. Next, the equivalence between Parts (a) and (b) follows from Lemma \ref{lem::4MM}. Now, we prove the equivalence between Parts (b) and (c). We note that the Nakayama functor~$\mathbf N$ preserves $\mathscr P(\mc{PI}_{\rm int}^{\mf p})$ and thus gives rise to a Serre functor on this category by \cite[Proposition 2.2]{MM}. From Propositions \ref{prop::twisting} and \ref{prop::completion}, we conclude that $\mathbf N$ induces the  functor $E_{\g}\otimes ({}_-)$ when restricted to $\mc{PI}_{\rm int}^{\mf p}$. This completes the proof. 
\end{proof}

 \begin{rem} We may note that $\mc O^\g =\mc O_{\rm int}^{\g}$ coincides with the category $\mc F$ of all finite-dimensional $\h_\oa$-semisimple $\g$-modules. Since all projective modules in $\mc F$ are injective, it follows that, as a special case of Theorem \ref{main::thm}, the Nakayama functor on $\mc F$ is always isomorphic to $E_\g\otimes ({}_-)$ when restricted to $\mc{PI}^{\g}$.
	\end{rem}

 \section{Examples: strange Lie superalgebras} \label{sect::pnqn}
 For given positive integers $m$ and $n$, recall that the general linear Lie superalgebra $\mathfrak{gl}(m|n)$ admits a realization as  $(m+n) \times (m+n)$ complex matrices
 \begin{align} \label{gllrealization}
 	\left( \begin{array}{cc} A & B\\
 		C & D\\
 	\end{array} \right),
 \end{align}
 where $A,B,C$ and $D$ are $m\times m, m\times n, n\times m, n\times n$ matrices,
 respectively. The Lie bracket of $\gl(m|n)$ is given by the super commutator. 	We define $E_{ij}$, for $1\leq i,j \leq m+n$, to be the elementary matrix in $\mathfrak{gl}(m|n)$ with  $(i,j)$-entry equal to $1$ and all other entries equal to $0$. 
 
  The goal of this section is to complete the proofs of Theorem~\ref{thmq::2} and Theorem~\ref{thmp::3}. We refer to  \cite[Section 1 and Section 2]{ChWa12}  	for further details about   the strange Lie superalgebras.

  	\subsection{Lie superalgebras of type P} 
 \label{subsect::pen} In this subsection, we set up the usual description of P-type Lie superalgebras and complete the proof of  Theorem \ref{thmp::3}.  Fix a positive integer $n$. 
 Then the {\em periplectic Lie superalgebra}~$\pn$  is a subalgebra of   $\mf{gl}(n|n)$ with the    following  matrix \mbox{realization} 
 \[ \mf{pe}(n)=
 \left\{ \left( \begin{array}{cc} A & B\\
 	C & -A^t\\
 \end{array} \right)\| ~ A,B,C\in \C^{n\times n},~\text{$B=B^t$ and $C=-C^t$} \right\}.
 \]  We may note that the even subalgebra $\g_\oa$ of $\g:=\pn$ is isomorphic to  $\gl(n)$. We also refer to \cite[Section 5]{CCC} for more details about the classification of parabolic decompositions  of $\pn$. Define the   Cartan subalgebra  $\mf h: = \bigoplus_{1\leq i \leq n}\C e_{ii}$, where  $e_{ii}: = E_{ii} -E_{n+i,n+i}$, for $1\leq i\leq n$. Let $\{\vare_i|~i=1, \ldots, n\}\subset \mathfrak{h}^*$ be the dual basis in $\h^\ast$. Then    $$E_{\pn}\otimes({}_-)\cong \Pi^{n}\C_{-2\omega_n}\otimes({}_-),$$
 where $\omega_n =\vare_1+\vare_2+\cdots+\vare_n$. The following corollary is a direct consequence of Corollary \ref{coro::111}. 
 
  \begin{cor} \label{cor::pn1}  Let $\mf p$ be a parabolic subalgebra of $\pn$ and $\Lambda\in \h^\ast/\Upsilon$.  Then the category $\mc{PI}_\Lambda^{\mf p}$ is not symmetric. 
 \end{cor}

Now we consider the  derived subalgebra  $\pn': = [\pn, \pn]$ of $\pn$, which is a  simple Lie superalgebra in the case when $n\geq 3$; see, e.g., \cite{Ka1}.  We may note that $\pn'_\oa \cong \mf{sl}(n)$ and $\pn'_\ob = \pn_\ob$. Therefore,  we have  $$E_{\pn'}\otimes({}_-)\cong \Pi^n({}_-).$$

The claims of Theorem \ref{thmp::3}  follow from the following theorem:
\begin{thm}  \label{thm::pn2} Let $\mf p$ be a parabolic subalgebra of $\pn'$  and  $\Lambda \in \h^\ast/\Upsilon$.
	\begin{itemize}
		\item[(a)] If $n$ is odd, then the category $\mc{PI}^{\mf p}_{\Lambda}$ is not symmetric.
		\item[(b)] The Serre functor on $\mathscr{P}(\mc{PI}^{\mf p}_{\rm int})$ is isomorphic to the functor $\Pi^n$. Furthermore, $\mc{PI}^{\mf p}_{\rm int}$ is symmetric if and only if $n$ is even.
	\end{itemize}  
\end{thm}
 \begin{proof}   Claim (a)  follows from Corollary \ref{coro::111}. The first assertion in Claim (b) follows from the proof of Theorem \ref{main::thm} mutatis mutandis, and this, together with Lemma \ref{lem::4MM}, imply the second assertion in Claim (b). 
 \end{proof}

  \subsection{Lie superalgebras of type Q} \label{subsect::qn}

  In this subsection, we fix a positive integer $n$. The {\em queer} Lie superalgebra $\mf q(n)$ can be realized as the following subalgebra of $\gl(n|n)$:
 	\[\mf g:=
 	\mf q(n)=
 	\left\{ \left( \begin{array}{cc} A & B\\
 		B & A\\
 	\end{array} \right) \| ~A,B\in \C^{n\times n} \right\}.
 	\]   
 	 We note that the even subalgebra of $\g$ is isomorphic to  $\gl(n)$.  Define the odd trace form 
	$$\text{otr}({}_-): \mf q(n)\rightarrow\C,\hskip0.2cm\left( \begin{array}{cc} A & B\\
		B & A\\
	\end{array} \right)\mapsto \text{tr}(B),$$ where $\text{tr}$ denotes  the usual trace of a matrix. Let $\text{Id}\in \mf q(n)$ be the identity matrix and define the subalgebra $\mf{sq}(n):=\{x\in \mf q(n)|~\text{otr}(x) =0\}$. Then the Q-type Lie superalgebras are given as follows:
	\begin{align}
	&\g =\g_n = \mf q(n),~\mf{sq}(n),~\mf{pq}(n): = \mf q(n)/(\text{Id}),\text{ and } \mf{psq}(n): = \mf{sq}(n)/(\text{Id}). \label{Qtype}
	\end{align} 

Fix a parabolic subalgebra $\mf p$ of $\mf g$. For each $\Lambda \in \h^\ast_\oa/\Upsilon$, recall that  we denote by $\mathbf{N}$  the Nakayama functor  on $\mc O^{\mf p}_\Lambda$.  The following lemma extends \cite[Corollary 5.12]{MM} to all Q-type Lie superalgebras $\g$ without any assumption on the dimension of $\g_\ob$.
   \begin{lem} \label{lem::ParCoApp}  Suppose that $\mf g_n$ is a Q-type Lie superalgebra from \eqref{Qtype}. Then we have 
   	  \begin{align*}
   		&\mathbf{N} \cong   	    \begin{cases} \Pi^n \circ C^2, &\mbox{ for~$  \g_n =\mf q(n)$ or $\mf{pq}(n)$;}\\  	    \Pi^{n-1} \circ C^2,&\mbox{ for $  \g_n = \mf{sq}(n)$ or $\mf{psq}(n)$}. \end{cases}
   	\end{align*}
   \end{lem}
\begin{proof} 
We may observe that the functor $E_\g\otimes({}_-): \mc O^{\mf p}_\Lambda\rightarrow \mc O^{\mf p}_\Lambda$ is isomorphic to $\Pi^n$ for $\g_n=\mf q(n),~\mf{pq}(n)$, and to $\Pi^{n-1}$ for $\g_n=\mf{sq}(n),~\mf{psq}(n)$. 
 In addition,	it follows from Theorem \ref{thm::realNak} that the functor    $\mathbf{N}$ naturally commutes with projective functors. The conclusion of the lemma follows from \cite[Theorem 5.1-(a)]{MM} mutatis mutandis the proof of \cite[Corollaries~4.6, 5.12]{MM}.
\end{proof}

	 The following theorem recovers the same results in \cite{MM} under the assumption that $\dim \g_\ob$ is even and implies the claim of Theorem~\ref{thmq::2} in Subsection \ref{sect::intro}.
	\begin{thm}
     Let $\g_n$ be a Q-type Lie superalgebra from \eqref{Qtype}. Then let $\mf p$ be a parabolic subalgebra
     	of $\g_n$ and let  $\Lambda\in \h^\ast_\oa/\Upsilon$.
     Then we have 
     \begin{itemize}
     	\item[(a)] If $\g_n =\mf q(n)$ or $\mf{pq}(n)$,  then  	$\mc{PI}_\Lambda^{\mf p}$ is symmetric if and only if $n$ is even.  Furthermore, the functor $\Pi^n$ is a Serre functor on $\mathscr{P}(\mc{PI}^\g)$. 
     	\item[(b)] If $\g_n = \mf{sq}(n)$ or $\mf{psq}(n)$, then  	$\mc{PI}_\Lambda^{\mf p}$ is symmetric if and only if $n$ is odd.  Furthermore, the functor $\Pi^{n-1}$ is a Serre functor on $\mathscr{P}(\mc{PI}^\g)$. 
     \end{itemize} 
	\end{thm}
	\begin{proof} The conclusion follows from Lemmata \ref{lem::4MM} and \ref{lem::ParCoApp}. 
	\end{proof}



\end{document}